\documentclass[11pt]{article}
\usepackage{definitions}
\usepackage{todonotes} %Added by Ben to tag some comments throughout, feel free to delete any time.
\usepackage{enumitem}

\begin{document}

\title{On the Linear Convergence of Extra-Gradient Methods for Nonconvex-Nonconcave Minimax Problems}
\author{Saeed Hajizadeh\thanks{Department of Mathematics, Statistics, and Computer Science, University of Illinois at Chicago ({mailto: shajiz2@uic.edu})}, Haihao Lu\thanks{The University of Chicago, Booth School of Business ({mailto: haihao.lu@uchicagobooth.com})}, Benjamin Grimmer\thanks{John Hopkins University, Department of Applied Mathematics and Statistics ({mailto: grimmer@jhu.edu})}}

\maketitle

\newtheorem{Def}{Definition}[section]
\newtheorem{Ass}{Assumption}[section]
\newtheorem{Thm}{Theorem}
\newtheorem{Prop}{Proposition}[section]
\newtheorem{Exc}{Exercise}
\newtheorem{Fact}{Fact}
\newtheorem{Obs}{Observation}[section]
\newtheorem{Lem}{Lemma}[section]
\newtheorem{Ex}{Example}[section]
\newtheorem{Cor}{Corollary}
\newtheorem{Rem}{Remark}

\begin{abstract}
    Recently, minimax optimization received renewed focus due to modern applications in machine learning, robust optimization, and reinforcement learning. The scale of these applications naturally leads to the use of first-order methods. However, the nonconvexities and nonconcavities present in these problems, prevents the application of typical Gradient Descent-Ascent, which is known to diverge even in bilinear problems. Recently, it was shown that the Proximal Point Method (PPM) converges linearly for a family of nonconvex-nonconcave problems. In this paper, we study the convergence of a damped version of Extra-Gradient Method (EGM) which avoids potentially costly proximal computations, only relying on gradient evaluation. We show that EGM converges linearly for smooth minimax optimization problem satisfying the same nonconvex-nonconcave condition needed by PPM.
\end{abstract}

\section{Introduction}
% \todo[inline]{Sean: A few journals to consider:
% \begin{itemize}
%     \item Computational optimization and applications
%     \item SIAM Journal on Control and Optimization
%     \item Journal of Optimization Theory and Applications
%     \item INFORMS Journal on optimization
%     \item Optimization Methods and Software
%     \item Foundations of Computational Mathematics
% \end{itemize}
% }

%The minimax problem arises when two competing agents seek to solve their problems when their interests not totally aligned, though, in some cases, such as Generative Adversarial Networks (GANs) \cite{Goodfellow:GAN}, the interests of the two involving agents are anti-aligned with each other. More precisely, consider a bifunction \[ L: \mathbb{R}^n \times \mathbb{R}^m \rightarrow \mathbb{R}  ~~~~ (x , y) \mapsto L(x , y), \]where agent $1$ tries to minimize with respect to one of the variables, $x$, and agent $2$ tries to maximize with respect to the other variable, $y$. In other words, we consider the following saddle point optimization problem 
Minimax optimization has a long history in optimization in the contexts of duality theory and robust optimization and more broadly in biology, social sciences, and economics \cite{RogerMyerson:GameTheory}. In these classic settings, the objective is typically well-behaved with respect to both the minimizing and maximizing variables (being convex and concave respectively) and strong computational guarantees are known. In the past few years, this problem has seen several applications in machine learning, and, in particular, in areas such as GANs \cite{Goodfellow:GAN}, adversarial training \cite{Madry:AdversarialTraining}, and multi-agent reinforcement learning \cite{OmidShafiei:ReinforcementLearning}. However, these modern applications do not possess the same convexity and concavity structure that is foundational to the traditional analysis of minimax optimization.
Formally, we consider unconstrained minimiax optimization problems of the following form
\begin{align} \label{generalproblem}
    \min_{x \in \mathbb{R}^n} \max_{y \in \mathbb{R}^m} L(x , y)\ ,
\end{align}
where $L$ is twice continuously differentiable but need not be convex nor concave in $x$ or $y$.

Due to the large-scale nature of such applications, first-order methods provide a practical, scalable algorithmic approach whereas higher-order methods may be computationally intractable. Most well-known first-order methods are Gradient Descent-Ascent (GDA), the Proximal Point Method (PPM), and the Extra Gradient Method (EGM). The most widely known of these methods is GDA and its iterate update, for solving problem (\ref{generalproblem}), with a step-size $s > 0$ is given by
\begin{align*}
    \begin{bmatrix}
        x_{k + 1} \\
        y_{k + 1}
    \end{bmatrix} = \begin{bmatrix}
        x_k \\
        y_k
    \end{bmatrix} - s \begin{bmatrix}
                         \nabla_x L(x_k , y_k) \\
                         - \nabla_y L(x_k , y_k)
                      \end{bmatrix} \ .
\end{align*}
However, GDA does not converge even in bilinear cases, i.e. the cases where the objective function $L(x , y) = x^T B y$ for some $B \in \mathbb{R}^{n \times m}$, without the assistance of averaging \cite{Daskalakis:LimitPointsNIPS,Lu:O(sr)}.

A more sophisticated algorithm to solve minimax problems, first studied in the seminal work of Rockafellar \cite{Rockafellar:PPM}, is the Proximal-Point Method (PPM). PPM enjoys more favorable theoretical properties compared to other first-order methods. The iteration update for PPM with step-size $s > 0$ is as follows
\begin{align*}
    \begin{bmatrix}
        x_{k + 1} \\
        y_{k + 1}
    \end{bmatrix} = {\prox}_{s.f}(x_{k},y_{k}) :=\argminmax_{u \in \mathbb{R}^n , v \in \mathbb{R}^m} L(u , v) + \frac{1}{2s} \| u - x_k \|^2 - \frac{1}{2s} \| v - y_k \|^2\ .
\end{align*}
Recently, it is shown in \cite{PPM_Landscape} that PPM (potentially with a damping term, see~\eqref{DampedPPMDefinition}) converges linearly to a stationary point of any nonconvex-nonconcave objective function whenever there is sufficient interaction between $x$ and $y$. However, despite these improved convergence guarantees, the proximal point method remains impractical to employ as it requires (approximately) solving the above proximal subproblem at every iteration. 

EGM replaces the proximal step of PPM with two steps in the direction of gradients evaluated at different points, as defined below:
\begin{equation*}
    \begin{aligned}
        \begin{bmatrix}
        x_{k + 1}' \\
        y_{k + 1}'
        \end{bmatrix} & = \begin{bmatrix}
        x_k \\
        y_k
        \end{bmatrix} - s \begin{bmatrix}
        \nabla_x L(x_k , y_k) \\
        - \nabla_y L(x_k , y_k)
        \end{bmatrix} & \\
        \begin{bmatrix}
        x_{k + 1} \\
        y_{k + 1}
        \end{bmatrix} & = \begin{bmatrix}
        x_k \\
        y_k
        \end{bmatrix} - s \begin{bmatrix}
        \nabla_x L\left(x_{k + 1}' , y_{k + 1}'\right) \\
        - \nabla_y L\left(x_{k + 1}' , y_{k + 1}'\right)
        \end{bmatrix} & \ .
    \end{aligned}
\end{equation*}
It is well known that EGM exhibits linear convergence to a saddle point for many convex-concave objective functions \cite{Nemirovski:EGM}. There are no such convergence guarantees for other inexpensive methods, vanilla GDA, which, as stated, is known to diverge even in bilinear case.

%Intuitively, given the nonconvex-nonconcave nature of the objective function, one acknowledges the damping parameter $\lambda$ as some sort of a slow and cautious advancement towards the solution. Despite all the fascinating theoretical properties and remarkable convergence results of damped PPM \cite{PPM_Landscape}, one notes that computing proximal step in (\ref{DampedPPMDefinition}) is challenging. In other words, PPM and all of its variations, such as damped PPM, are considered to be \emph{implicit methods}. To avoid such computational barriers, one ought to look into gradient-based methods such as GDA. One natural method is generalized EGM where one would take an aggressive first step to find the middle point followed by a cautious step after ``assessing the situation''.

In this paper, we show that under similar conditions to PPM, EGM (as well as its generalization incorporating damping) globally converges linearly to a stationary point of a class of nonconvex-nonconcave objective functions. This algorithm is computationally more efficient to implement than PPM, relying only on a gradient oracle. Hence, we find that similar to the convex concave setting of~\cite{Nemirovski:EGM}, EGM provides an efficient replacement for the proximal point method for nonconvex-nonconcave optimization while preserving its strong theoretical guarantees.

The rest of the paper is organized as follows. In section \ref{Literature}, the relevant literature is provided. In Section \ref{Notation}, the notation will be laid out for the rest of the paper. In section \ref{PreliminarySection}, the preliminaries and the required definitions are furnished. In Section \ref{EGMApproximationOfPPM}, we will present and explain our damped generalization of EGM and show that it approximates a damped variant of PPM. Section \ref{MainResult} formalizes and proves our main Theorem and then provides a nonconvex-nonconcave quadratic minimax problem establishing the tightness of our convergence theory.

% \section{Notation and Literature} 

% In this section, we furnish our notation for the technical details of the paper and also the assumptions. Literature is also provided in this section.

\subsection{Related Literature} \label{Literature}
Proximal Point Method (PPM) was first introduced by Rockafellar in his seminal work \cite{Rockafellar:PPM} as a method to solve variational inequality problems. Minimax problems are a special class of variational inequalities and the results Rockafellar established in his paper imply the local linear convergence of PPM when applied to (\ref{generalproblem}) granted (i) the solution to (\ref{generalproblem}) is unique, (ii) the mapping $F:\mathbb{R}^{m \times n} \rightarrow \mathbb{R}^{m \times n}$ is invertible around $0$, and (iii) $F^{-1}$ is Lipschitz continuous around $0$. Paul Tseng \cite{Tseng:PPM_EGM} showed the convergence, under some complicated conditions, of PPM and EGM. Nemirovski \cite{Nemirovski:EGM}, in his seminal paper, introduces Mirror Prox algorithm, a generalization of EGM,
and shows that it converges linearly with rate $O(\frac{1}{\epsilon})$ for convex-concave minimax problems. 

There are quite a few works that study a special case of problem (\ref{generalproblem}) with bilinear interaction terms, i.e. when $L(x , y) = f(x) + x^TAy - g(y)$ where $A \in \mathbb{R}^{n \times m}$ is a matrix and $f,g$ are two proper, lower semi-continuous, and convex functions. The most well-known such work is Nesterov's smoothing \cite{Nesterov:NonsmoothMinimization}, Monteiro's Hybrid Proximal Extragradient Method \cite{Monteiroetal:HybridProximalEGM}, Douglas-Rachford Splitting Method \cite{DouglasRachford}, Primal-Dual Hybrid Gradient Method (PDHG) \cite{ChambollePock:PDHGIntro}, and Alternating Direction Method of Multipliers (ADMM) \cite{Jonathan_Bertsekas:OnDouglasRachford}. Recently, O`Connor-Vandenberghe \cite{OConnor:PDHGandDRSplittingEquivalence} showed that PDHG and Douglas-Rachford Splitting Method are equivalent. ADMM and PDHG differ from EGM in the sense that the formers update the primal and dual sequentially while EGM updates the primal and dual iterates simultaneously.

The introduction of General Adversarial Networks (GANs) \cite{Goodfellow:GAN} has shifted a ton of attention towards nonconvex-nonconcave minimax problems. Recently, Grimmer et. al \cite{PPM_Landscape} studied PPM and its global convergence on a class of nonconvex-nonconcave minimax problems that enjoy sufficiently strong interaction between the two underlying agents. Yang et al. \cite{Yangetal:NonconvexNonconcaveMinimaxGlobalConvergence} furnished another example of global convergence of Alternating Gradient Descent-Ascent (AGDA) for the class of objective functions that satisfy the \emph{two-sided Polyak-Łojasiewicz} inequality, known to be a weaker condition than strongly-convex-strongly-concave. \cite{Jin_etal:localOptimalityMinimax} studied the various notions of optimality in nonconvex-nonconcave minimax problems. Diakonikolas et. al \cite{Diakonikolasetal:NonconvexNonconcaveErgodicConvergenceEGM} furnished the proof for the ergodic convergence of a special class of our damped EGM, that is, one with a damping parameter of $\lambda = \frac{1}{2}$, when applied to a nonconvex-nonconcave minimax problems. Lee and Kim \cite{FastEGMLee} furnished the convergence of the so-called \emph{two-time-scale anchored extra-gradient method} (FEG) to a stationary point (Definition \ref{StationaryPoint}) of an objective function with a negatively $\rho$-comonotone oracle. This is also known in literature as $|\rho|$-cohypomonotonicity (see \cite{WalaaMoursi:MonotoneOperators}, Remark $2.5$ (ii)). Each iteration of FEG iterates from a convex combination of the ``current iterate'' $z_k$ and the ``initial point'' $z_0$, and moves along the direction of a linear combination of gradient information in its transition from the mid-point to the next iteration. This was shown to converge sublinearly to a stationary point of an objective function that admits negative comonotonicity. The convergence rate of various implicit and explicit methods for solving nonconvex-nonconcave minimax problems, their convergence rate, are illustrated in Table \ref{ComparisonOfRates}. We note that positive interaction dominance, a slight strengthening of negative comonotonicity, is, to the best of the authors' knowledge, the most general setting for which linear convergence is known in the literature.

\begin{table}
    \centering
    \captionsetup{justification=centering}
    \begin{tabular}{ p{3.5cm}||p{1.2cm}| p{6cm}| p{2cm}| p{1.5cm}  }
        \multicolumn{5}{c}{} \\

         Algorithm & Explicit Method & Setting & Constraints & Convergence Rate  \\
         \hline
        AGDA~\cite{Yangetal:NonconvexNonconcaveMinimaxGlobalConvergence} & \checkmark & two-sided PŁ & \ding{56} & $O\left( \log \left( \frac{1}{\epsilon} \right) \right)$  \\
        PPM~\cite{PPM_Landscape} & \ding{56} & Positive Interaction Dominance  & \ding{56} & $O\left( \log \left( \frac{1}{\epsilon} \right) \right)$ \\
        Damped EGM~\cite{Diakonikolasetal:NonconvexNonconcaveErgodicConvergenceEGM} & \checkmark & weak MVI  & \ding{56} & $O\left( \frac{1}{\epsilon} \right)$ \\
        CEG+~\cite{Pethick:EscapingLimitCycle} & \ding{56} & weak MVI  & \checkmark & $O\left( \frac{1}{\epsilon} \right)$ \\
        EGM Variant~\cite{FastEGMLee} & \checkmark & Negatively Comontone  & \ding{56} & $O\left( \frac{1}{\epsilon} \right)$ \\
        \hline
        Damped EGM[\textbf{This paper}] & \checkmark & Positive Interaction Dominance  & \ding{56} & $O\left( \log \left( \frac{1}{\epsilon} \right) \right)$ \\
         \hline
    \end{tabular}
    \caption{Comparison of the algorithms studied in recent papers on nonconvex-nonconcave minimax optimization. Any method that used a resolvent or proximal step is not considered an explicit method and may thus have an expensive cost per iteration. }
    \label{ComparisonOfRates}
\end{table}
% \todo[inline]{Ben: EG+ is an explicit method, right? It actually appears to be identical to the damped method we consider. Their + refers to allowing the intermediate step to be longer (which is equivalent to having the second step be shorter, relatively, as we do). I believe our positioning should be that we are showing EG+ converges linearly under positive interaction. (Complementing their existing result that EG+ converges sublinearly under nonnegative interaction.)}

% \todo[inline]{Saeed: You're right! They wrote their algorithm iteration in their paper as an $\argmin$ subproblem (which I was surprised initially when I was reading it). I just took another look, their argmin has a closed form solution, whence an explicit method which is almost similar to damped EGM. I have focuesd on the linear convergence issues more.}

% \todo[inline]{Ben: Since EG+ and (Damped) EGM are the same method, I think we should have their naming in the table reflect that. In my view, [7] and [15] showed (Damped) EGM converges sublinearly under weak MVI and Nonnegative Interaction Dominance, respectively. We show (Damped) EGM converges linearly under Positive Interaction Dominance.}

Last but not least, \cite{ZhangNon-convexSmoothGames} provided a unified approach to various local and global optimal notions such as local and global Nash equilibrium as well as local and global minimax point with enlightening examples throughout. \cite{ZhangNon-convexSmoothGames} also furnished two preliminary results on the local linear convergence (asymptotic stability) condition of damped EGM, i.e. with a fixed damping constant, for smooth nonconvex-nonconcave minimax games but furnished no explicit rates.

The literature on convex-concave minimax problems, however, is much richer. Among those, Daskalakis et. al \cite{Daskalakisetal:OGDA} studied Optimistic Gradient Descent Ascent (OGDA) for training GANs, showing that ODGA converges linearly to the unique stationary point of a bilinear objective function with the interaction matrix being nonsingular. Mokhtari et. al \cite{Mokhtari:UnifiedMinimax} furnished a unified study of OGDA, EGM, and PPM, and showed that OGDA and EGM approximate PPM. They also re-proved the concept of linear convergence of EGM, OGDA, and PPM to the unique stationary point of strongly-convex-strongly-concave and full-rank bilinear minimax problems. 

\subsection{Notations and Assumptions} \label{Notation}

Throughout the paper, we use $x$ and $y$ to denote the minimizing and maximizing variables, respectively. We use $I$ to denote identity matrix of appropriate dimension. The symbols $\nabla$, $\nabla^2$, $\nabla_x$, and $\nabla_{xx}^2$ are used to denote the gradient, Hessian, partial gradient, and partial Hessian of a function following the symbol. Let $U \subset \mathbb{R}^n \times \mathbb{R}^m$. We say a mapping $L:\mathbb{R}^n \times \mathbb{R}^m \rightarrow \mathbb{R}$ is $\xi$-Lipschitz on $U$ if for any pair of $z , z' \in U$, $\| L(z) - L(z') \| \leq \xi \| z - z' \|$. We are primarily interested in twice differentiable functions $L$. We say $L$ is $\beta$-smooth on $U$ if its gradient is $\beta$-Lipschitz in $U$, or equivalently, when $L$ is twice continuously differentiable and its Hessian satisfies $\| \nabla^2 L(z) \| \leq \beta$ for all $z \in U$. We say that $L$ is $\mu$-strongly convex-strongly concave on $U$ for $\mu>0$ if for any $z = (x , y) \in U$, $\nabla_{xx}^2 L(z) \succeq \mu I$ and $- \nabla_{yy}^2L(z) \succeq \mu I$. When $\mu = 0$, this is equivalent to convexity and concavity in $x$ and $y$, respectively. However, our primary interest is in nonconvex-nonconcave objectives. We quantify the level of negative curvature in $x$ and positive curvature in $y$ as follows: We say $L$ is $\rho$-weakly convex-weakly concave on $U$ if for all $z = (x , y) \in U$, \[ \nabla_{xx}^2 L(z) \succeq -\rho I,~~~~~~ - \nabla_{yy}^2L(z) \succeq -\rho I\ . \]
Moreover, we denote the first-order oracle for the problem~\eqref{generalproblem} by $ F(z)=\begin{bmatrix}
\nabla_xL(z) \\
- \nabla_yL(z)
\end{bmatrix}$ for any $z\in\mathbb{R}^n\times\mathbb{R}^m$. %The operation $\otimes_l$ represents the $l^{th}$-order tensor product. In particular, for example, $\nabla^2 F(z) \otimes_2 F(z)$ is the tensor product of the second order tensor $\nabla^2 F(z)$ with the vector field $F(z)$.
For any $\rho$-weakly-convex-weakly-concave function $L$, we denote the prox operator with stepsize $0<s\leq 1/\rho$ by
\begin{align*}
    (x , y) = {\prox}_{s.f} (u , v) := \argminmax_{u \in \mathbb{R}^n , v \in \mathbb{R}^m} f(u , v) + \frac{1}{2s} \| u - x \|^2 - \frac{1}{2s} \| v - y \|^2 \ .
\end{align*}

%Finally, $T:\mathbb{R}^n \rightrightarrows \mathbb{R}^m$ denotes to a set-valued mapping $T$, where $T(x)$ is a subset of $\mathbb{R}^m$.

For any $0<s\leq 1/\rho$, we say a function $L$ is $\alpha(s)$-interaction dominant in $x$ if for all $z \in \mathbb{R}^{n + m}$, 
\begin{align}
    \nabla_{xx}^2L(z) + \nabla_{xy}^2L(z)\left(s^{-1} I - \nabla_{yy}^2L(z)\right)^{-1}\nabla_{yx}^2L(z) & \succeq \alpha(s) I \label{InteractionDominanceX}
\end{align}
and $\alpha(s)$-interaction dominant in $y$ if for any $z \in \mathbb{R}^{n + m}$
\begin{align}
    -\nabla_{yy}^2L(z) + \nabla_{yx}^2L(z)\left(s^{-1} I + \nabla_{xx}^2L(z)\right)^{-1}\nabla_{xy}^2L(z) & \succeq \alpha(s) I \label{InteractionDominanceY}\ . 
\end{align}

Throughout our analysis of the Extragradient Method, we assume the following regularity conditions on the objective. The first two conditions (Lipschitz continuity, smoothness, and weak convexity-concavity) are relatively standard. The third regularity condition is positive interaction dominance (see Assumption~\ref{Assumptions}) and is equivalent to the settings considered for the proximal point method in~\cite{PPM_Landscape} and the negative comonotone setting where accelerated, sublinear rates were recently derived by~\cite{FastEGMLee}.

\begin{Ass}\label{Assumptions}
    The objective function $L\colon\mathbb{R}^n\times\mathbb{R}^m \rightarrow \mathbb{R}$ satisfies the following four conditions
    \begin{enumerate}
        %\item $L$ is $\xi$-Lipschitz on $\mathbb{R}^n\times\mathbb{R}^m$.
        \item $L$ is continuously twice differentiable and $\beta$-smooth on $\mathbb{R}^n\times\mathbb{R}^m$.
        \item $L$ is $\rho$-weakly convex in $x$ and $\rho$-weakly concave in $y$ on $\mathbb{R}^n\times\mathbb{R}^m$.
        \item For some $s \in \left( 0 , \frac{1}{\rho} \right)$, $L$ satisfies the interaction dominance conditions~\eqref{InteractionDominanceX} and~\eqref{InteractionDominanceY} where $\alpha > 0$ is a positive function.
    \end{enumerate}
    The smoothing constant $\beta$, weak-convexity-weak-concavity constant $\rho$, and a pair of values $s$ and $\alpha$ satisfying interaction dominance are needed to select stepsize parameters where our theory applies.
\end{Ass}

\paragraph{Positive Interaction-Dominance:} Provided $0<s\leq 1/\rho$, the second terms 
\begin{align*}
    \nabla_{xy}^2L(z)\left(s^{-1} I - \nabla_{yy}^2L(z)\right)^{-1}\nabla_{yx}^2L(z) \\
    \nabla_{yx}^2L(z)\left(s^{-1} I + \nabla_{xx}^2L(z)\right)^{-1}\nabla_{xy}^2L(z)    
\end{align*}
in \eqref{InteractionDominanceX} and \eqref{InteractionDominanceY} are always positive semi-definite. We thus see that any convex-concave objective function is always nonnegative interaction dominant in both $x$ and $y$. A $\rho$-weakly-convex-weakly-concave function is $\alpha(s)$-interaction dominant with $\alpha(s) \geq -\rho$. For this weakly-convex-weakly-concave function $L(x , y)$ to have nonnegative interaction dominance, $L$ must have ``large enough'' interaction between $x$ and $y$ in the Hessian in order to ``overcome'' the effect of the smallest (negative) eigenvalue of partial Hessian $\nabla_{xx}^2 L$ and the largest (positive) eigenvalue of  partial Hessian $\nabla_{yy}^2L$.

\section{Preliminaries}  \label{PreliminarySection}

Grimmer et al. \cite{PPM_Landscape} used a generalization of the Moreau envelope, called \emph{saddle envelope}, introduced by \cite{Wets_et.al:SaddleEnvelope}, to study the behavior of a certain class of nonconvex-nonconcave objective functions. More precisely, given any $\rho$-weakly-convex-weakly-concave objective function $L: \mathbb{R}^n \times \mathbb{R}^m \rightarrow \mathbb{R}$ with $L \in \mathcal{C}^2$, and any $s > 0$, the saddle envelope $\mathcal{L}_s$ is defined as
\begin{align} \label{SaddleEnvelope}
    \mathcal{L}_s(x , y) := \min_{u \in \mathbb{R}^n} \max_{v \in \mathbb{R}^m} \left\{ L(u , v) + \frac{1}{2s} \| u - x \|^2 - \frac{1}{2s} \| v - y \|^2  \right\} \ .
\end{align}

Suppose $s < \frac{1}{\rho}$, then it is easy to see that $M(u , v) := L(u , v) + \frac{1}{2s} \| u - x \|^2 - \frac{1}{2s} \| v - y \|^2$ is $\left(\frac{1}{s} - \rho \right)$-strongly-convex-strongly-concave so that the saddle envelope (\ref{SaddleEnvelope}) is well-defined. Corollary $2.2$ in \cite{PPM_Landscape} implies that to find \emph{any} stationary point of $L$, one only needs to find those of $\mathcal{L}_s$. Imposing nonnegative interaction dominance paves the way for proving more powerful properties of the saddle envelope.

If an objective function is nonnegative interaction dominant in both $x$ and $y$ for some $s \in \left( 0 , \frac{1}{\rho} \right)$, its saddle envelope $\mathcal{L}_{s}(x , y)$ would then be convex-concave (\cite{PPM_Landscape}, Proposition $2.6$). If for some $s \in \left( 0 , \frac{1}{\rho} \right)$ the objective function is further positive interaction dominance, then the saddle envelope $\mathcal{L}_{s}(x , y)$ is strongly-convex-strongly-concave.  
% \footnote{This upper bound will be re-stated and justified in Remark \ref{ExplicitBoundsOnEta}}. 

These interaction dominance conditions can be equivalently characterized in terms of the convexity and concavity of the saddle envelop~\eqref{SaddleEnvelope} and in terms of the monotonicity of the saddle gradient $F(z) = \begin{bmatrix}
        \nabla_xL(z) \\
        -\nabla_y L(z)
    \end{bmatrix}$. This is formalized in Proposition \ref{EquivalenceOfStatements}.
    
Recently, \cite{FastEGMLee} presented algorithms with sublinear rate for nonconvex-nonconcave minimax optimization problems with $\rho$-comonotone gradient oracle for some negative $\rho$. A set-valued mapping $T:\mathbb{R}^{n + m} \rightrightarrows \mathbb{R}^{n + m}$ is said to be \textbf{$\rho$-monotone} on $\mathbb{R}^{n + m}$ if for every $\bar{z} \in \mathbb{R}^{n + m}$ there exists a neighborhood $V$ of $\bar{z}$ such that 
\begin{align*}
    \left( w - w' \right)^T \left( z - z' \right) \geq \rho \| z - z' \|^2 ~~~~~~ \text{for all} ~ z , z' \in V \text{and all} ~ w \in T(z), ~ w' \in T\left( z' \right)\ .
\end{align*}

A set-valued mapping\footnote{The use of set-valued operators is typical here, allowing these definitions to be applied to cases where the objective function is not smooth whence only admitting subdifferentials. However, such nonsmooth optimization is beyond the scope of our paper.} $T:\mathbb{R}^{n + m} \rightrightarrows \mathbb{R}^{n + m}$ is said to be \textbf{$\rho$-comonotone} if for every $\bar{z} \in \mathbb{R}^{n + m}$ there exists a neighborhood $V$ of $\bar{z}$ such that 
\begin{align*}
    \left( w - w' \right)^T \left( z - z' \right) \geq \rho \| w - w' \|^2 ~~~~~~ \text{for all} ~ z , z' \in V \text{and all} ~ w \in T(z), ~ w' \in T\left( z' \right)\ .
\end{align*}
The following proposition, in particular, establishes the equivalence of our assumptions and the negative comonotone setting recently considered by~\cite{FastEGMLee}.
    
    %pair of propositions which extend the argument in Proposition $2.6$ of \cite{PPM_Landscape}.
\begin{Prop} \label{EquivalenceOfStatements}
    Let $L(x , y)$ be a twice-differentiable, $\rho$-weakly-convex-weakly-concave objective function, and $s \in \left( 0 , \frac{1}{\rho} \right)$. Then the following statements are equivalent:
    \begin{enumerate} [label = (\roman*)]
        \item $L(x , y)$ is $\alpha(s) \geq 0$-interaction dominance in both $x$ and $y$,
        \item The saddle envelope $\mathcal{L}_s(x , y)$ of $L$ is convex-concave,
        \item The oracle $F(x , y) = \begin{bmatrix}
            \nabla_x L(x , y) \\
            \nabla_y L(x , y)
        \end{bmatrix}$ of $L(x , y)$ is $-s$-comonotone.
    \end{enumerate}
\end{Prop}

\begin{proof}
    $(i) \iff (ii)$: Observe that the conditions~(\ref{InteractionDominanceX}) and (\ref{InteractionDominanceY}) hold with nonnegative $\alpha$ exactly when 
    \begin{equation*} 
        \begin{aligned}
            \nabla_{xx}^2L(z) + \nabla_{xy}^2L(z)\left(s^{-1} I - \nabla_{yy}^2L(z)\right)^{-1}\nabla_{yx}^2L(z) & \succeq 0 & \\
            -\nabla_{yy}^2L(z) + \nabla_{yx}^2L(z)\left(s^{-1} I + \nabla_{xx}^2L(z)\right)^{-1}\nabla_{xy}^2L(z) & \succeq 0\ . 
        \end{aligned}
    \end{equation*}
    Adding an identity matrix $\frac{1}{s} I$ above and inverting the resulting positive definite matrix yields the following equivalent characterization
    \begin{equation*} 
        \begin{aligned}
            s^{-1} I - s^{-2} \left( s^{-1} I + \nabla_{xx}^2L (z) + \nabla_{xy}^2L(z)\left(s^{-1} I - \nabla_{yy}^2L(z)\right)^{-1}\nabla_{yx}^2L(z) \right)^{-1} \succeq 0 & \\
            s^{-1} I - s^{-2} \left( s^{-1} I - \nabla_{yy}^2L (z) + \nabla_{yx}^2L(z)\left(s^{-1} I + \nabla_{xx}^2L(z)\right)^{-1}\nabla_{xy}^2L(z) \right)^{-1} \preceq 0 . 
        \end{aligned}
    \end{equation*}
    By Lemma $3$ in \cite{PPM_Landscape}, these are exactly the $xx$ and $yy$ components of the saddle envelope's Hessian, i.e. $\nabla^2_{xx} \mathcal{L}_s(z)$ and $\nabla^2_{yy} \mathcal{L}_s(z)$, whence the assertion is proved.
    
    $(ii) \iff (iii)$: In~\cite[Appendix A.1]{FastEGMLee}, Lee and Kim showed that the saddle envelope's gradient mapping $F_s(.) = \begin{bmatrix}
        \nabla_x \mathcal{L}_{s}(.) \\
        -\nabla_y \mathcal{L}_{s}(.)
    \end{bmatrix}$ is monotone if and only if $F$ is $-s$-comonotone. Recalling that a gradient mapping $(\nabla_x L,-\nabla_y L)$ is monotone if and only if the associated function is convex-concave, the proof is complete.
\end{proof}

Let us now state the definition of a stationary point of a bifunction.

\begin{Def} \label{StationaryPoint}
    A point $(x^* , y^*) \in \mathbb{R}^{n + m}$ is a stationary point of an objective function $L(x , y)$ if 
    \[ \nabla_x L(x^* , y^*) = 0 , ~~~ \text{and} ~~~ \nabla_y L(x^* , y^*) = 0\ . \] 
\end{Def}
We now provide a result that is all but stated in \cite{PPM_Landscape}:
\begin{Lem} \label{UniqueStationaryPoint}
    A $\rho$-weakly-convex-weakly-concave objective function $L$ that is positive interaction dominance in both $x$ and $y$ has a unique stationary point.
\end{Lem}

\begin{proof}
    By hypothesis and Proposition $2.6$ in \cite{PPM_Landscape}, the saddle envelope is strongly-convex-strongly-concave whence adopting a unique saddle point $(x^* , y^*)$ which is also its unique stationary point. By Corollary $2.2$ in \cite{PPM_Landscape}, then $(x^* , y^*)$ is the unique stationary point of $L$; because otherwise, if $L$ has any other stationary point $(\tilde{x}^* , \tilde{y}^*) \neq (x^* , y^*)$, it would clearly contradict Corollary $2.2$ of \cite{PPM_Landscape}. 
\end{proof}

\section{Damped EGM} \label{EGMApproximationOfPPM}
Damped PPM, as the name suggests, introduces a damping parameter $\lambda \in (0 , 1]$ in each iteration of PPM. The iteration update is
\begin{align} \label{DampedPPMDefinition}
    \begin{bmatrix}
        x_{k + 1} \\
        y_{k + 1}
    \end{bmatrix} = (1 - \lambda) \begin{bmatrix}
        x_k \\
        y_k
    \end{bmatrix} + \lambda ~ \prox_{sf} \left( \begin{bmatrix}
    x_k \\
    y_k
    \end{bmatrix} \right) \ .
\end{align}
This damping, illustrated in (\ref{DampedPPMDefinition}), decreases the size of the proximal step in each iteration. The inclusion of $\lambda = 1$ allows taking a full proximal step in each iteration.

In this section, we introduce the damped EGM and show that it is approximating the damped PPM introduced in  \cite{PPM_Landscape}. We first recall that EGM is an approximation of PPM \cite{Nemirovski:EGM}. Next, we will extend this result and show that the damped EGM is an approximation to damped PPM~\eqref{DampedPPMDefinition}. Notice that damped EGM does not need to solve an implicit step, thus the update is computationally cheaper than that of damped PPM \eqref{DampedPPMDefinition}.

The damped EGM is presented in Algorithm \ref{GeneralizedEGM}.

\begin{algorithm}
    \caption{Damped EGM}\label{GeneralizedEGM}
    \hspace*{\algorithmicindent} \textbf{Input}: $z_0 := (x_0 , y_0)$, step-size $s > 0$, damping parameter $\lambda \in (0 , 1]$, and tolerance $\epsilon > 0$ \\
    % \hspace*{\algorithmicindent} \textbf{Output} 
    \begin{algorithmic}[1]
        \State $k=0$
        \While {$\left \lVert \begin{bmatrix}
        \nabla_x L(x_k , y_k) \\
        - \nabla_y L(x_k , y_k)
        \end{bmatrix} \right\rVert \geq \epsilon$}
        \State \hspace*{\algorithmicindent} Find the mid-point: $\begin{bmatrix}
        x_{k + 1}' \\
        y_{k + 1}'
        \end{bmatrix} = \begin{bmatrix}
        x_k \\
        y_k
        \end{bmatrix} - s \begin{bmatrix}
        \nabla_x L(x_k , y_k) \\
        - \nabla_y L(x_k , y_k)
        \end{bmatrix}$,
        \State \hspace*{\algorithmicindent}  Find the next-iterate: $\begin{bmatrix}
        x_{k + 1} \\
        y_{k + 1}
        \end{bmatrix} = \begin{bmatrix}
        x_k \\
        y_k
        \end{bmatrix} - \lambda s \begin{bmatrix}
        \nabla_x L\left(x_{k + 1}' , y_{k + 1}'\right) \\
        - \nabla_y L\left(x_{k + 1}' , y_{k + 1}'\right)
        \end{bmatrix}$
        \State \hspace*{\algorithmicindent} $k \leftarrow k+1$
        \EndWhile
    \end{algorithmic}
\end{algorithm}

Damped EGM is a generalization of traditional EGM, where the step-size of the two steps can be chosen differently. This difference in steps comes from the damping parameter $\lambda$ that controls the length of the second step. Intuitively, it is natural to think that the lack of convexity and concavity would require an algorithm to take a type of ``cautiously aggressive'' steps to avoid divergence and cycling that are common phenomena in nonconvex-nonconcave minimax optimization \cite{Grimmer:LimitingBehaviorMinimax,PPM_Landscape}. The following Proposition establishes that the damped EGM in Algorithm \ref{GeneralizedEGM} approximates damped PPM.

% It is virtually a two-step gradient descent ascent. Despite such fascinating simplicity almost akin to that of GDA, we will find this algorithm efficient in solving certain classes of nonconvex-nonconcave minimax problems. In observing so, we first note that EGM is an approximation of PPM \cite{Nemirovski:EGM}. In the following section, we extend this approximation and show that our generalized EGM approximates damped PPM.

% \subsection{Generalized EGM Approximates Damped PPM} \label{AEGMDPPMApproximation}
% Let us first review damped PPM in more detail as this is necessary for a complete explanation of the approximation.

\begin{Prop}
    Damped EGM update in Algorithm \ref{GeneralizedEGM}, when applied to the $\rho$-weakly-convex-weakly-concave objective function $L(x , y)$, is an approximation to the update for damped PPM (\ref{DampedPPMDefinition}).
\end{Prop}
\begin{proof}
    We write the Taylor expansion of the update for $x$ in  Algorithm \ref{GeneralizedEGM}, which gives us,
\begin{align}
    x_{k + 1} & = x_k - \lambda s \nabla_x L\left(x_k - s \nabla_x L(z_k) , y_k + s \nabla_y L(z_k) \right) & \nonumber \\
    & = x_k - \lambda s \left[ \nabla_x L(z_k) - s \nabla_{xx}^2L(z_k) \nabla_x L(z_k) + s \nabla_{xy}^2 L(z_k) \nabla_y L(z_k) + o(s) \right] & \nonumber \\
    & = x_k - \lambda s \nabla_x L(z_k) + \lambda s^2 \nabla_{x x}^2L(z_k) \nabla_x L(z_k) - \lambda s^2 \nabla_{xx}^2L(z_k) \nabla_y L(z_k) + o\left(s^2\right)\ . & \label{generalizedEGMUpdateX}
\end{align}
Similarly, one can find the update of $y$ as 
\begin{align}
    y_{k + 1} & = y_k + \lambda s \nabla_y L(z_k) - \lambda s^2 \nabla_{yx}^2L(z_k) \nabla_x L(z_k) + \lambda s^2 \nabla_{yy}^2L(z_k) \nabla_yL(z_k) + o\left(s^2\right)\ . & \label{generalizedEGMUpdateY}
\end{align}
Let now $z_k^+ := \prox_{s.L} (z_k)$ be one proximal step of size from the current iterate $z_k$. We then have,
\begin{align}
    x^+_k & = x_k - s \nabla_xL \left( x_k - s \nabla_xL(z^+_k) , y_k + s \nabla_yL(z^+_k) \right) & \nonumber \\
    & = x_k - s \left[ \nabla_xL(z_k) - s \nabla_{xx}^2L(z_k) \nabla_xL(z^+_k) + s\nabla_{xy}^2L(z_k) \nabla_yL(z^+_k) + o(s) \right] & \nonumber \\
    & = x_k - s \nabla_xL(z_k) + s^2 \nabla_{xx}^2L(z_k) \nabla_xL(z_k) - s^2 \nabla_{xy}^2L(z_k) \nabla_yL(z_k) + o(s^2)\ . & \label{PPMUpdateX}
\end{align}
where the second equality follows from the local Lipschitzness of partial gradients.

Similarly,
\begin{align} 
    y^+_k = y_k + s \nabla_yL(z_k) - s^2 \nabla_{yx}^2L(z_k) \nabla_xL(z_k) + s^2 \nabla_{yy}^2L(z_k) \nabla_yL(z_k) + o(s^2)\ . \label{PPMUpdateY}
\end{align}

We know from (\ref{DampedPPMDefinition}) that the update iterate of a damped PPM with constant $\lambda$ is given by 
\begin{align} \label{DampedPPM}
    \tilde{z}_{k + 1} = (1 - \lambda) z_k + \lambda z_k^+\ .
\end{align}
Combining (\ref{PPMUpdateX}) and (\ref{PPMUpdateY}) with (\ref{DampedPPM}) and comparing with (\ref{generalizedEGMUpdateX}) and (\ref{generalizedEGMUpdateY}), one can observe that
\begin{align*}
    \| \tilde{z}_{k + 1} - z_{k + 1} \| = o\left( s^2 \right)\ .
\end{align*}
This concludes the proof of the claim that the damped EGM update as defined in Algorithm \ref{GeneralizedEGM} is an approximation to that of damped PPM.
\end{proof}

\section{Main Result} \label{MainResult}
In this section, we show that if $L(x,y)$ is interaction dominance, then the damped EGM with proper step-size converges linearly to a stationary point. Furthermore, we show that damped EGM may diverge without interaction dominance, which showcases the tightness of using interaction dominance to characterize the convergence of damped EGM.

% state our main result. It furnishes the linear convergence of generalized EGM to the stationary point of a nonconvex-nonconcave objective function $L(x , y)$ with nonnegative interaction baseline in both $x$ and $y$. We will furnish the proof of this Theorem following all the remarks on and ramifications of the Theorem. Last but not least, we will furnish an example that establishes the tightness of our main Theorem.

\subsection{Convergence Result} \label{ConvergenceResult}   
First, let us state our main convergence result:
\begin{Thm} \label{GeneralizedEGMConvergenceTheorem}
    Suppose the objective function $L:\mathbb{R}^n \times \mathbb{R}^m \rightarrow \mathbb{R}$ in Problem (\ref{generalproblem}) satisfies Assumption \ref{Assumptions}, and let $z^* := (x^* , y^*)$ be its saddle point. Choose the parameters $s$ and $\lambda$ such that they satisfy 
    \begin{align} \label{ConditionsInTheorem}
        \frac{2}{s^3\left(\frac{1}{s \alpha(s)} + 1\right)} > \beta^3 ,~~~~~~~\lambda < \min\left\{ 1 , \left( \frac{1}{s \rho} - 1 \right)^2 \right\} \left[ \frac{2}{\frac{1}{s \alpha(s)} + 1} -  s^3 \beta^3 \right]
    \end{align}

    The damped EGM with step-size $s \in \left( 0 , \frac{1}{\rho} \right)$ applied to the problem $\min_{x \in \mathbb{R}^n} \max_{y \in \mathbb{R}^m} L(x , y)$ with the constant $\lambda \in (0 , 1]$ linearly converges to the unique stationary point of $L$. More precisely, for any iterate $(x_k , y_k)$ and starting point $(x_0 , y_0)$ one has 
    \begin{align*}
        \left\lVert \begin{bmatrix}
            x_k - x^* \\
            y_k - y^*
        \end{bmatrix} \right\rVert \leq \left( 1 - \frac{2 \lambda}{\frac{1}{s \alpha(s)} + 1} + \frac{\lambda^2}{\min\left\{ 1 , \left( \frac{1}{s\rho} - 1 \right)^2 \right\}} + \lambda s^3 \beta^3 \right)^k \left\lVert \begin{bmatrix}
            x_0 - x^* \\
            y_0 - y^*
        \end{bmatrix} \right\rVert.
    \end{align*}
\end{Thm}

We now state some of remarks to better understand the statement of the theorem, simplify the conditions under which the statements hold, and observe what the Theorem translates into when considering special cases,

% Below Et = O( \sqrt{\frac{\beta}{\mu}} \beta)
\begin{Rem}
    For any $\mu$-strongly-concave-strongly-concave, Theorem \ref{GeneralizedEGMConvergenceTheorem} has a linear convergence rate of $O\left( \frac{1}{s \lambda \mu} \log \left( \frac{1}{\epsilon} \right) \right)$. This is evident by plugging $\alpha = \mu$ and $\rho = -\mu$ in the convergence rate in the statement of the Theorem. Assuming that $s = O\left( \beta^{-1} \sqrt{\frac{\mu}{\beta}} \right)$, which results in a step-size smaller than $\frac{1}{\beta}$ by a factor of $\sqrt{\frac{\mu}{\beta}}$, we obtain a convergence rate of $O\left( \sqrt{\frac{\beta^3}{\lambda^2 \mu^3}} \log \left( \frac{1}{\epsilon} \right) \right)$ which is a reasonable linear rate. The best rate known for EGM in the convex-concave case is $O\left( \frac{\beta}{\mu} \log \left( \frac{1}{\epsilon} \right) \right)$ (see for example~\cite[Theorem~1]{Gidel2018Variational}, \cite[Theorem~7]{Mokhtari:UnifiedMinimax}, \cite[Lemma~3.1]{Tseng:PPM_EGM}, or \cite[Proposition~2.2]{AlvezMonteiroetal:HPEMIP})
\end{Rem}

\begin{Rem} \label{ParameterSelection}
    The selection of the pair of parameters $\lambda$ and $s$ satisfying (\ref{ConditionsInTheorem}) given $\rho$ and $\beta$, and $\alpha$ as a function of $s$ is not difficult. One observes that, plugging $s = \frac{t}{\beta}$ in the condition on $s$ on the left-hand side of (\ref{ConditionsInTheorem}), one would get an inequality in $t$ and solve for $t$. In many cases, e.g. quadratic problems, this inequality entails solely a polynomial and is trivial to solve. This would enlighten one on how smaller than $\beta$ should the step-size $s > 0$ be taken. We would like to point out that the possible values for $s$ usually involve an interval as opposed to arbitrarily small values. We also note that the damping introduced in our method is sometimes necessary for the convergence of EGM.

% Below $ET \in (153.5 , 407.1)$
    For instance, consider the problem 
    \begin{align} \label{BenNonconvexProblem}
        \min_{x} \max_{y} L(x , y) = f(x) + \bar{A} x y - f(y), ~~~~~~ \text{with} ~ f(x) = (x-1)(x+1)(x-3)(x+3)
    \end{align}
    that is a nonconvex-nonconcave problem. Let $\bar{A} = 100$. A simple examination of our conditions in the Theorem as described above implies that any $s \in (0.00245 , 0.00651)$  with a damping factor $\lambda \in (0 , 0.06)$ would guarantee convergence. Choosing $s^* = 200$ and $\lambda^* = 0.01$ we observe convergence for any starting point the box $[-4 , 4] \times [-4 , 4]$ as in Fig. \ref{fig:Nonconvex-NonconcaveBen1}. Choosing, instead, $\lambda^* = 1$, whence recovering undamped EGM, results in cycling as shown in Fig. \ref{fig:Nonconvex-NonconcaveBen2}.
    % We will cast more light on this in Remark \ref{DampingNecessary}.
\end{Rem}

\begin{Rem} \label{ExplicitBoundsOnEta}
    Let us see through the inequality on the left-hand side of (\ref{ConditionsInTheorem}) more explicitly under some assumptions. Suppose we are given a problem that is nonnegative interaction dominant and we further restrict $s$ to satisfy
    \begin{align} \label{EtaLowerBound}
        \frac{1}{s} \geq \frac{\max \left\{ \| \nabla_{xx}^2L \|^2 , \| \nabla_{xx}^2L \|^2 \right\} }{\rho} > \frac{\rho^2}{\rho} = \rho\ .
    \end{align} 
    On the other hand, we have
    \begin{align*}
        \left( s^{-1} I - \nabla_{yy}^2L(z) \right)^{-1} & \succeq \left( s^{-1} + \max \left\{ \| \nabla_{xx}^2L \| , \| \nabla_{xx}^2L \| \right\} \right)^{-1} I  \\
        \left( s^{-1} I + \nabla_{xx}^2L(z) \right)^{-1} & \succeq \left( s^{-1} + \max \left\{ \| \nabla_{xx}^2L \| , \| \nabla_{xx}^2L \| \right\} \right)^{-1} I \ .
    \end{align*}
    so that $\alpha$ can be lower-bounded as
    \begin{align} \label{AlphaLowerBound}
        \alpha \geq -\rho + \frac{s . \lambda_{ \min } \left( \nabla_{xy}^2L(z) \nabla_{yx}^2L(z) \right) }{1 + s . \max \left\{ \| \nabla_{xx}^2L \| , \| \nabla_{xx}^2L \| \right\} }\ .
    \end{align}
    Clearly, $s$ can not be chosen arbitrarily small as that would mar interaction dominance. From (\ref{AlphaLowerBound}), a sufficient upper bound on $s$ to preserve nonnegative interaction dominance is
    \begin{align} \label{EtaUpperBound1}
        s \geq \frac{\rho}{ \lambda_{ \min } \left( \nabla_{xy}^2L(z) \nabla_{yx}^2L(z) \right) - \rho . \max \left\{ \| \nabla_{xx}^2L \| , \| \nabla_{xx}^2L \| \right\}\ } .
    \end{align}
    The lower bound (\ref{EtaUpperBound1}), as mentioned, preserves nonnegative interaction dominance. It further illustrates what the ``sufficiently large'' interaction requirement means. To further illustrate the explicit restrictions on $s$, in light of the inequality on the left-hand side of (\ref{ConditionsInTheorem}), let us further suppose $s < \frac{1}{\beta}$\footnote{This condition is consistent with the classical smooth optimization literature choosing a step-size smaller than the reciprocal of the Lipschitz constant of the oracle.}. This further assumption and (\ref{AlphaLowerBound}) imply
    \begin{align} \label{AlphaLowerBound2}
        \alpha(s) \geq -\rho + \frac{1 + s. \max \left\{ \| \nabla_{xx}^2L \| , \| \nabla_{xx}^2L \| \right\}}{s . (\beta + \max \left\{ \| \nabla_{xx}^2L \| , \| \nabla_{xx}^2L \| \right\} )}\cdot \rho = \frac{1 - s \cdot \beta}{s \cdot (\beta + \max \left\{ \| \nabla_{xx}^2L \| , \| \nabla_{xx}^2L \| \right\}) }\cdot \rho \ .
    \end{align}
    Combining (\ref{EtaLowerBound}) and (\ref{AlphaLowerBound2}) with the inequality on the left-hand side of (\ref{ConditionsInTheorem}) one would get
    \begin{align*}
        \frac{2}{s^3\left(\frac{1}{s \alpha(s)} + 1\right)} & \geq \frac{2 \frac{\max \left\{ \| \nabla_{xx}^2L \|^6 , \| \nabla_{xx}^2L \|^6 \right\}}{\rho^3}}{1 + \frac{ (\beta + \max \left\{ \| \nabla_{xx}^2L \| , \| \nabla_{xx}^2L \| \right\})}{\rho(1 + s \cdot \beta)}} \geq \frac{2 \frac{\max \left\{ \| \nabla_{xx}^2L \|^6 , \| \nabla_{xx}^2L \|^6 \right\}}{\rho^3}}{1 + \frac{1}{s \cdot \rho}} > \beta^3. 
    \end{align*}
    Therefore, the lower bound
    \begin{align} \label{EtaUpperBound2}
        s > \frac{1}{\rho \left[ 2 \left( \frac{\max \left\{ \| \nabla_{xx}^2L \|^2 , \| \nabla_{yy}^2L \|^2 \right\}}{\rho \beta} \right)^3 - 1 \right]}  .
    \end{align} is a sufficient to guarantee the inequality condition on the left-hand side of (\ref{ConditionsInTheorem}) is satisfied.
    
    The reader notes all at once that (\ref{EtaUpperBound1}) and (\ref{EtaUpperBound2}) are two \textbf{explicit} lower bounds on $s$ illustrating how small could one select the step-size while ensuring the conditions of the Theorem, nonnegative interaction dominance and the inequality condition on the left-hand side of (\ref{ConditionsInTheorem}) are satisfied. We have been conservative to preserve the implicit condition in the main theorem to preserve generality as much as is achievable.
\end{Rem}

\begin{Rem}
    One observes that damped EGM has a slower convergence rate $1 - \frac{2 \lambda}{\frac{1}{s \cdot \alpha(s)} + 1} + \frac{\lambda^2}{\min\left\{ 1 , \left( \frac{1}{s \rho} - 1 \right)^2 \right\}} + \lambda s^3 \beta^3$ than damped PPM introduced in \cite{PPM_Landscape} which converges at a rate of $1 - \frac{2 \lambda}{\frac{1}{s \alpha(s)} + 1} + \frac{\lambda^2}{\min\left\{ 1 , \left( \frac{1}{s \rho} - 1 \right)^2 \right\}}$. This difference aligns with the perspective of EGM approximating the proximal step via two cheaper gradient descent-ascent steps with accuracy depending on the smoothness of the objective function $\beta$. 
\end{Rem}

\begin{figure} 
    \begin{subfigure}{.5\textwidth}
        \centering
        \includegraphics[width=0.8\linewidth]{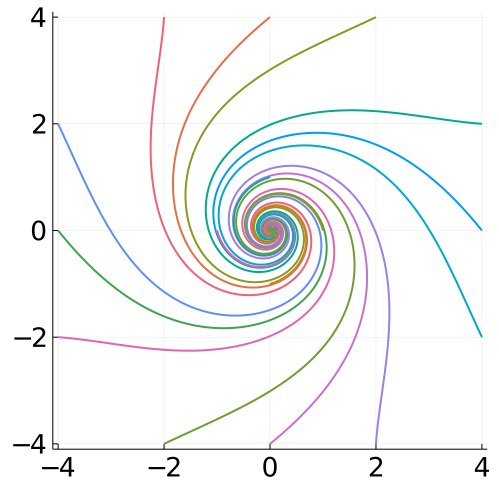}
	    \caption{Convergence of damped EGM}
	    \label{fig:Nonconvex-NonconcaveBen1}
    \end{subfigure}%
    \begin{subfigure}{.5\textwidth}
        \centering
        \includegraphics[width=0.8\linewidth]{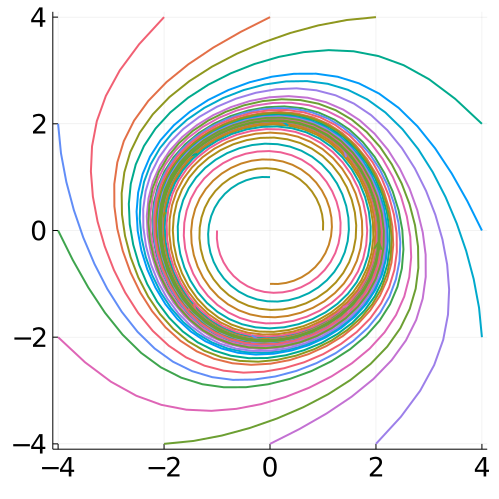}
	    \caption{Cycling of vanilla EGM}
	    \label{fig:Nonconvex-NonconcaveBen2}
    \end{subfigure}%
    \caption{The convergence of damped EGM to the unique stationary point of the nonconvex-nonconcave problem (\ref{BenNonconvexProblem}) with $\bar{A} = 100$ from any starting point within the box $[-4 , 4] \times [-4 , 4]$ and the cycling of EGM.}
    \label{fig:Nonconvex-NonconcaveBen}
\end{figure}

\begin{Rem}
    Our main result also recovers the setting and results of \cite{ZhangNon-convexSmoothGames}. In particular, Theorem $5.14$ in \cite{ZhangNon-convexSmoothGames} asserts that one does not lose convergence by shrinking the damping parameter. This fact follows as well from our theorem, which additionally quantifies the rate of convergence rate one would see as the damping parameter shrinks.
\end{Rem}

\subsection{Proof of Theorem \ref{GeneralizedEGMConvergenceTheorem}} \label{MainResultProof}
Having noticed the remarks and ramifications of our main theorem, we now furnish a proof for the theorem.

\begin{proof}[Proof of Theorem~\ref{GeneralizedEGMConvergenceTheorem}]
    Given any iterate $z_k$, let us first find the exact upper bound to $\| \tilde{z}_{k +1} - z_{k + 1} \|$ where, as noted before, $\tilde{z}_{k + 1}$ is the update of damped PPM and $z_{k + 1}$ is the update of the damped EGM. One notes that 
\begin{align*}
    z_{k + 1} = z_k - \lambda s F\left( z_k - s F(z_k) \right)\ ,
\end{align*}
and
\begin{align*}
    \tilde{z}_{k + 1} & = (1 - \lambda)z_k + \lambda z_k^+ = (1 - \lambda)z_k + \lambda \left[ z_k - s F\left( z_k^+ \right) \right]  = z - \lambda s F\left( z_k^+ \right)\ .
\end{align*}
One also would observe that $z_k^+ = z_k - s F\left( z_k^+ \right)$, whence $\left( I + s F \right) \left( z_k^+ \right) = z_k$. We need the following lemma in order for further proceeding with the proof.

\begin{Lem} \label{Invertibility}
    The operator $I + s F:\mathbb{R}^{n+m} \rightarrow \mathbb{R}^{n + m}$ for any $0 < s < \frac{1}{\rho}$ is invertible. 
\end{Lem}

\begin{proof}
    Since $L$ is $\beta$-smooth, the operator $(I + s F)(.)$ is continuous. On the other hand, at any $z \in \mathbb{R}^{n + m}$ we have 
    \begin{align}
        \left\lvert I + s \nabla F(z) \right\lvert & = \left\lvert \begin{bmatrix}
    I + s \nabla_{xx}^2L(z) & \nabla_{xy}^2L(z) \\
    - \nabla_{xy}^2L(z)^T & I - s \nabla_{yy}^2L(z)
    \end{bmatrix} \right\lvert & \nonumber \\
    & = \left\lvert I + s \nabla_{xx}^2L(z) \right\lvert. \left\lvert I - s \nabla_{yy}^2L(z) + \nabla_{xy}^2L(z)^T(I + s \nabla_{xx}^2L(z))^{-1}\nabla_{xy}^2L(z) \right\lvert \ , \nonumber
    \end{align}
    where the second equality follows from Schur complement. Moreover, by hypothesis, $I + s \nabla_{xx}^2L(z) \succ 0$, $I - s \nabla_{yy}^2L(z) \succ 0$, and \[ \nabla_{xy}^2L(z)^T(I + s \nabla_{xx}^2L(z))^{-1}\nabla_{xy}^2L(z) \succeq 0\ , \]whence the Jacobian has always nonzero determinant, i.e. $\left\lvert I + s \nabla F(z) \right\lvert \neq 0$. Therefore, by inverse function Theorem, the operator $(I + s F)(.)$ is invertible with a continuously differentiable inverse.
\end{proof}

One observes that the proximal step in (\ref{DampedPPM}) is equivalent to that of \cite{PPM_Landscape} with $\eta = \frac{1}{s}$. For any $z$, the inverse $(I + s F)^{-1}(z)$ of the operator $I + s F$ is given by
\begin{align} \label{InversePPMOperator}
    (I + s F)^{-1}(z) = z - s F(z) + s^2 \nabla F(z) F(z) + o\left( s^2 \right) \ .
\end{align}
For details, one may refer to Appendix B of \cite{Lu:O(sr)}. Therefore, we can write
\begin{align} \label{IntegralRemainderBound}
    \| \tilde{z}_{k + 1} - z_{k + 1} \| & = \lambda s \| F\left( (I + s F)^{-1}(z_k) \right) - F\left( z_k - s F(z_k) \right) \| & \nonumber \\
    & \leq \lambda s \beta \left\| \frac{1}{1!} \int_{0}^{s} \frac{\partial^2}{\partial \tau^2}(I + \tau F)^{-1}(z_k)\bigg|_{\tau = t} ~  (s - t)  ~ dt \right\| & \nonumber \\
    & \leq \lambda s \beta  \int_{0}^{s} \left\| \frac{\partial^2}{\partial \tau^2}(I + \tau F)^{-1}(z_k)\bigg|_{\tau = t}\right\| (s - t) ~ dt \ ,
\end{align}
where $\beta$-smoothness of $L$, Taylor expansion with the integral form of the remainder of the inverse operator $(I + s F)^{-1}(z)$, and Cauchy-Schwartz inequality are invoked. 

We now turn to evaluate $\frac{\partial^2}{\partial \tau^2}(I + \tau F)^{-1}(z_k)\bigg|_{\tau = t}$. For that matter, first note that by Appendix B in \cite{Lu:O(sr)} and by Lemma \ref{Invertibility}, we have for any $t < \frac{1}{\rho}$
\begin{align*} 
    (I + tF)^{-1}(z) = z - tF(z) + t^2 \nabla F(z) F(z) + t^3 \left( - (\nabla F(z))^2F(z) - \frac{1}{2} \nabla^2F(z) \otimes_2 F(z) \right) + o\left( t^3 \right).
\end{align*}
where $\otimes_2$ refers to the $2$-times tensor product of a $2$-dimensional tensor with a vector. Now for any $z$, let $h_z : \mathbb{R} \rightarrow \mathbb{R}^{n + m}$ be the mapping $h_z: t \mapsto (I + tF)^{-1}(z)$. By Lemma \ref{Invertibility}, for any $t < \frac{1}{\rho}$ the derivative,
\begin{align*}
    h_z^{'}(t) & = \frac{\partial}{\partial t}(I + t F)^{-1}(z) & \nonumber \\
    & = - F(z) + t(2 \nabla F(z) F(z)) + t^2 \left( - 3(\nabla F(z))^2F(z) - \frac{3}{2} \nabla^2F(z) \otimes_2 F(z) \right)+ o\left( t^2 \right) &
\end{align*}
is continuous. Invoking, in addition, the mean value Theorem, for any $t < \frac{1}{\rho}$, one would have a well-defined mapping $f:(0 , t] \rightarrow (0 , t]$, $f: t \mapsto f(t) \in (0 , t)$ such that
\begin{align*}
    h_z^{'}(t) = - F(z) + 2 f(t) \nabla F(z) F(z)\ .
\end{align*}
This mapping $f$ is continuous in $(0 , \frac{1}{\rho})$, because otherwise there would exist an $\epsilon > 0$, $t_0 \in (0 , \frac{1}{\rho})$ such that for any $\delta > 0$,
\begin{align*}
    \inf_{t \in \mathbb{B}(t_0 , \delta) \setminus \{ t_0 \}} |h^{'}_z(t) - h^{'}_z(t_0)| = \inf_{t \in \mathbb{B}(t_0 , \delta) \setminus \{ t_0 \}} |2 \nabla F(z) F(z)|\cdot|f(t) - f(t_0)| = |2 \nabla F(z) F(z)|\cdot\epsilon \ ,
\end{align*}
which contradicts the continuity of $h_z^{'}(t)$.

Hence, for any $t < 1/\rho$, there exists a sequence $\delta_n \downarrow 0$ such that
\begin{align*}
    h_z^{''}(t) & = \limsup_{n \rightarrow \infty} \frac{h_z^{'}(t + \delta_n) - h_z^{'}(t)}{\delta_n}  = 2\nabla F(z) F(z) \limsup_{n \rightarrow \infty} \frac{f(t + \delta_n) - f(t)}{\delta_n} \ ,
\end{align*}
with $\limsup_{n \rightarrow \infty} \frac{f(t + \delta_n) - h_z^{'}(t)}{\delta_n} \leq f(t) \leq 1$ by construction. Therefore, we can bound (\ref{IntegralRemainderBound}) as follows
\begin{align*}
    \| \tilde{z}_{k + 1} - z_{k + 1} \| & \leq \lambda s^3 \beta \| \nabla F(z_k) \| . \| F(z_k) \| \leq \lambda s^3 \beta^3 \| z_k - z^* \| 
\end{align*}
by using $F(z^*) = 0$.

Let $c := 1 - \frac{2 \lambda}{\frac{1}{s \cdot \alpha(s)} + 1} + \frac{\lambda^2}{\min\left\{ 1 , \left( \frac{1}{s \rho} - 1 \right)^2 \right\}} \in (0 , 1)$ be the shrinking constant of the distance to the unique stationary point of $L$ on one iteration of damped PPM as in \cite{PPM_Landscape}. Given the upper bound for $\| \tilde{z}_{k + 1} - z_{k + 1} \|$ just furnished, one has
\begin{align*}
    \| z_{k + 1} - z^* \|^2 & = \| z_{k + 1} - \tilde{z}_{k + 1} \|^2 + \| \tilde{z}_{k + 1} - z^* \|^2 + 2(z_{k + 1} - \tilde{z}_{k + 1})^T(\tilde{z}_{k + 1} - z^*) & \nonumber \\
    & \leq \left( \lambda^2 \beta^6 s^6 + c^2 + 2 \lambda c s^3 \beta^3 \right) \| z_k - z^* \|^2 & \nonumber \\
    & = \left( \lambda \beta^3 s^3 + c \right)^2 \| z_k - z^* \|^2 \ .
\end{align*}
Should one select a value of $s$ and $\lambda$ such that 
\begin{align*}
    \lambda \beta^3 s^3 + c < 1
\end{align*}
linear convergence can be claimed. One would attain convergence if $\lambda$ is chosen small enough, that is,
\begin{align} \label{BoundOnLambda}
    \lambda < \min\left\{ 1 , \left( \frac{1}{s \rho} - 1 \right)^2 \right\} \left[ \frac{2}{\frac{1}{s \alpha(s)} + 1} - s^3 \beta^3 \right] \ .
\end{align}
One notices that (\ref{BoundOnLambda}) indicates an implicit lower bound on $s$, that is already needed to be smaller than $\frac{1}{\beta}$. Given $\alpha(s) > 0$, one would need $s$ to be small enough to make the upper bound on $\lambda$ positive (i.e. so that for some $\lambda \in (0 , 1)$ convergence can be attained). 
\end{proof}

\subsection{Tightness of Nonnegative Interaction Dominance} \label{Examples}

We observed that damped EGM converges to the saddle point of an objective function in problem \eqref{generalproblem} if $L$ is nonnegative interaction dominance in both variables $x$ and $y$. This interaction dominance does not hold if $s$ is too small. This is an interesting observation which is in contradiction with classic minimization problems where every step-size smaller than the reciprocal of the smoothing modulus is acceptable for convergence.  An interesting question, therefore, is whether there is a class of nonconvex-nonconcave problems for which it is ``necessary'' to have interaction dominance in both variables for the convergence to the saddle point of \eqref{generalproblem}. We answer this question in the affirmative by exploring the class of nonconvex-nonconcave quadratic saddle problems and showing that the nonnegative interaction dominance is necessary for convergence so that our main result in Theorem \ref{GeneralizedEGMConvergenceTheorem} is tight. This would affirm that nonconvex-nonconcave minimax problems are in contrast to classical optimization problems where choosing any step-size smaller than the inverse of the Hessian norm would suffice to guarantee convergence.

We show that a slight nonconvexity-nonconcavity in a given quadratic saddle problem would necessitate the nonnegative, in fact, even positive, interaction dominance to hold for convergence to occur. More precisely, consider the following nonconvex-nonconcave quadratic problem with interaction $\bar{A}$\footnote{For simplicity, we are considering equal dimensions for both $x$ and $y$.},
\begin{align} \label{Nonconvex-NonconcaveQuadraticProblem}
    \min_{x \in \mathbb{R}^n} \max_{y \in \mathbb{R}^n} L(x , y) := -\frac{\rho}{2} x^Tx + \bar{A} x^Ty + \frac{\rho}{2}y^Ty\ .
\end{align}

It is observed in \cite{PPM_Landscape} that for the very specific problem of quadratic minimax optimization problem (\ref{Nonconvex-NonconcaveQuadraticProblem}), interaction dominance is a necessary condition for the convergence of damped PPM on that problem. Since damped EGM and damped PPM only differ in terms concerning derivatives of higher order, it is natural to think that damped EGM too accedes the positive interaction dominance as a necessary condition in converging to the solution of (\ref{Nonconvex-NonconcaveQuadraticProblem}). We show that this, indeed, is the case. One can observe this by plugging the objective function of problem (\ref{Nonconvex-NonconcaveQuadraticProblem}) in the update iterations (\ref{generalizedEGMUpdateX})-(\ref{generalizedEGMUpdateY}) of damped EGM. More precisely, first notice that the largest $\alpha$ that satisfies the interaction dominance conditions (\ref{InteractionDominanceX})-(\ref{InteractionDominanceY}) for the problem (\ref{Nonconvex-NonconcaveQuadraticProblem}) is $\alpha = -\rho + \frac{s \cdot \bar{A}^2}{1 - s \cdot \rho}$. The update on $x$ is given by
\begin{align*}
    x_{k + 1} & = x_k - \lambda s \nabla_x L(z_k) + \lambda s^2 \nabla_{xx}^2L(z_k) \nabla_x L(z_k) - \lambda s^2 \nabla_{xy}^2L(z_k) \nabla_y L(z_k) & \\
    & = \left( 1 + s \lambda \rho + s^2 \lambda \rho^2 - \lambda s^2 \bar{A}^2 \right) x_k - \left( s \lambda \bar{A} + 2 s^2 \lambda \rho \bar{A} \right) y_k \ .
\end{align*}
Applying similar calculations for $y_{k + 1}$ and stacking the equations yields,
\begin{align} \label{MatrixOfGeneralizedEGM}
    \begin{bmatrix}
    x_{k + 1} \\
    y_{k + 1}
    \end{bmatrix} = \begin{bmatrix}
    \Theta I & -\Sigma I \\
    \Sigma I & \Theta I
    \end{bmatrix}\begin{bmatrix}
    x_k \\
    y_k
    \end{bmatrix} \ ,
\end{align}
where $\Theta := \left( 1 + s \lambda \rho + s^2 \lambda \rho^2 - \lambda s^2 \bar{A}^2 \right)$ and $\Sigma := s \lambda \bar{A} \left( 1 + 2 s \rho \right)$.  Taking the norm of both sides in (\ref{MatrixOfGeneralizedEGM}) and simplifying, one gets
\begin{align*}
    \| z_{k + 1} \|^2 = (\Theta^2 + \Sigma^2) \| z_k \|^2 \ .
\end{align*}

Now observing the unique stationary point of the objective function of the quadratic problem (\ref{Nonconvex-NonconcaveQuadraticProblem}) is $z = 0$, it follows that damped EGM converges if and only if $\Theta^2 + \Sigma^2 < 1$ holds. 

Suppose now that in problem (\ref{Nonconvex-NonconcaveQuadraticProblem}) we have $\bar{A} = 10$ and $\rho > 0$ a very small positive value. In other words, the problem is nonconvex-nonconcave with a small negative curvature in partial Hessian $\nabla_{xx}^2L(z)$ and a small positive curvature in partial Hessian $\nabla_{yy}^2L(z)$. Therefore, one can write the convergence condition of EGM on problem (\ref{Nonconvex-NonconcaveQuadraticProblem}) as follows
\begin{align*}
    \Theta^2 + \Sigma^2 = \left( 1 + s \lambda \rho + \lambda s^2 (\rho^2 - 100) \right)^2 + \left( 10 \lambda s \left( 1 + 2 s \rho \right) \right)^2 < 1  .
\end{align*}
It is easy to notice that one must have $s \lambda \rho + \lambda s^2 (\rho^2 - 100) < 0$ for convergence because otherwise the term $\Theta > 1$ so that $\Theta^2 + \Sigma^2 > 1$, and by (\ref{MatrixOfGeneralizedEGM}) the iterations diverge away from the origin. This implies \[ s > \frac{\rho}{100 - \rho^2}\ . \]Since $\alpha(s) = -\rho + \frac{100 s}{1 - s \rho}$ the convergence condition on $\alpha(s)$ simplifies to 
\begin{align} \label{QuadraticAlphaPositive}
    \alpha(s) > \frac{\rho^2 (\rho + 1)}{100 - 2 \rho^2}\ .
\end{align}
The condition (\ref{QuadraticAlphaPositive}) implies that even a small nonconvexity-nonconcavity in the problem would necessitate the value of $\alpha(s)$ to be positive in order to attain convergence. Hence even for simple quandratic minimax optimization, our guarantees based on the positive interaction dominance condition are essentially tight, as simple counter examples exist just beyond this regime.

% \textcolor{red}{We now consider the more general class of quadratic example:
% \begin{align}
%     \min_{x \in \mathbb{R}^n} \max_{y \in \mathbb{R}^m} \frac{1}{2} x^T P x + x^T B y + \frac{1}{2} y^T Q y,
% \end{align}
% where the matrices $P \in \mathbb{R}^{n \times n}$, $Q \in \mathbb{R}^{n \times n}$, and $B \in \mathbb{R}^{n \times m}$, are symmetric. This problem has a solution at the origin. Writing down the update equation for the middle point and next iteration for damped EGM we have,
% \begin{align}
%     \begin{bmatrix}
%         x_{k + 1} \\
%         y_{k + 1}
%     \end{bmatrix}
%     = \left( I - \lambda s \begin{bmatrix}
%         P + s B B^T - s P^2 & B + s BQ - s PB  \\
%         - B^T - s Q B^T + s B^T P & - Q + s B^T B - s Q^2
%     \end{bmatrix} \right) 
%     \begin{bmatrix}
%         x_k \\
%         y_k
%     \end{bmatrix}
% \end{align}
% Granted the damping parameter is small enough, we see that a sufficient condition for the convergence, from any initial point, of the damped EGM to the solution $(0 , 0)$ is 
% \begin{align*}
%     P + s B B^T \succ s P^2 \succeq 0 \\
%     - Q + s B^T B \succ s Q^2 \succeq 0
% \end{align*}
% which is a similar condition to positive interaction dominance.  \textbf{I do not think this is going anywhere!}} 

\section{Conclusion}
In this paper, we showed that EGM with scaled steps also called damped EGM approximates damped PPM. This algorithm is simpler to implement and computationally more efficient than damped PPM. We showed that damped EGM linearly converges to the saddle point of any nonconvex-nonconcave minimax problems that satisfy the nonnegative interaction dominance condition. We also furnished an example displaying the approximate tightness of nonnegative interaction dominance condition.

\bibliographystyle{amsplain}
\bibliography{refs}

\end{document}